\documentclass[11pt,leqno]{article}
\usepackage[margin=1in]{geometry} 
\usepackage{amssymb,amsfonts,amsmath,bbm,mathrsfs,stmaryrd,mathtools}
\usepackage{xcolor}
\usepackage{url}

\usepackage{extarrows}

\usepackage[shortlabels]{enumitem}
\usepackage{tensor}

\usepackage{xr}
\usepackage[T1]{fontenc}

\usepackage[colorlinks,
linkcolor=black!75!red,
citecolor=blue,
pdftitle={},
pdfproducer={pdfLaTeX},
pdfpagemode=None,
bookmarksopen=true,
bookmarksnumbered=true,
backref=page]{hyperref}

\usepackage{tikz}
\usetikzlibrary{arrows,calc,decorations.pathreplacing,decorations.markings,intersections,shapes.geometric,through,fit,shapes.symbols,positioning,decorations.pathmorphing}

\usepackage{braket}

\usepackage[amsmath,thmmarks,hyperref]{ntheorem}
\usepackage{cleveref}

\newcommand\nthalias[1]{\AddToHook{env/#1/begin}{\crefalias{lemma}{#1}}}

\nthalias{definition}
\nthalias{example}
\nthalias{examples}
\nthalias{remark}
\nthalias{remarks}
\nthalias{convention}
\nthalias{notation}
\nthalias{construction}
\nthalias{theoremN}
\nthalias{propositionN}
\nthalias{corollaryN}
\nthalias{lemma}
\nthalias{proposition}
\nthalias{corollary}
\nthalias{theorem}
\nthalias{conjecture}
\nthalias{question}
\nthalias{assumption}

\creflabelformat{enumi}{#2#1#3}

\crefname{section}{Section}{Sections}
\crefformat{section}{#2Section~#1#3} 
\Crefformat{section}{#2Section~#1#3} 

\crefname{subsection}{\S}{\S\S}
\AtBeginDocument{%
  \crefformat{subsection}{#2\S#1#3}%
  \Crefformat{subsection}{#2\S#1#3}%
}

\crefname{subsubsection}{\S}{\S\S}
\AtBeginDocument{%
  \crefformat{subsubsection}{#2\S#1#3}%
  \Crefformat{subsubsection}{#2\S#1#3}%
}

\theoremstyle{plain}

\newtheorem{lemma}{Lemma}[section]
\newtheorem{proposition}[lemma]{Proposition}
\newtheorem{corollary}[lemma]{Corollary}
\newtheorem{theorem}[lemma]{Theorem}

\theoremstyle{plain}
\theoremnumbering{Alph}
\newtheorem{theoremN}{Theorem}

\theoremstyle{plain}
\theorembodyfont{\upshape}
\theoremsymbol{\ensuremath{\blacklozenge}}

\newtheorem{definition}[lemma]{Definition}
\newtheorem{example}[lemma]{Example}

\newtheorem{remark}[lemma]{Remark}
\newtheorem{remarks}[lemma]{Remarks}

\crefname{definition}{definition}{definitions}
\crefformat{definition}{#2definition~#1#3} 
\Crefformat{definition}{#2Definition~#1#3} 

\crefname{ex}{example}{examples}
\crefformat{example}{#2example~#1#3} 
\Crefformat{example}{#2Example~#1#3} 

\crefname{exs}{example}{examples}
\crefformat{examples}{#2example~#1#3} 
\Crefformat{examples}{#2Example~#1#3} 

\crefname{remark}{remark}{remarks}
\crefformat{remark}{#2remark~#1#3} 
\Crefformat{remark}{#2Remark~#1#3} 

\crefname{remarks}{remark}{remarks}
\crefformat{remarks}{#2remark~#1#3} 
\Crefformat{remarks}{#2Remark~#1#3} 

\crefname{convention}{convention}{conventions}
\crefformat{convention}{#2convention~#1#3} 
\Crefformat{convention}{#2Convention~#1#3} 

\crefname{notation}{notation}{notations}
\crefformat{notation}{#2notation~#1#3} 
\Crefformat{notation}{#2Notation~#1#3} 

\crefname{table}{table}{tables}
\crefformat{table}{#2table~#1#3} 
\Crefformat{table}{#2Table~#1#3}

\crefname{lemma}{lemma}{lemmas}
\crefformat{lemma}{#2lemma~#1#3} 
\Crefformat{lemma}{#2Lemma~#1#3} 

\crefname{proposition}{proposition}{propositions}
\crefformat{proposition}{#2proposition~#1#3} 
\Crefformat{proposition}{#2Proposition~#1#3} 

\crefname{propositionN}{proposition}{propositions}
\crefformat{propositionN}{#2proposition~#1#3} 
\Crefformat{propositionN}{#2Proposition~#1#3} 

\crefname{corollary}{corollary}{corollaries}
\crefformat{corollary}{#2corollary~#1#3} 
\Crefformat{corollary}{#2Corollary~#1#3} 

\crefname{corollaryN}{corollary}{corollaries}
\crefformat{corollaryN}{#2corollary~#1#3} 
\Crefformat{corollaryN}{#2Corollary~#1#3} 

\crefname{theorem}{theorem}{theorems}
\crefformat{theorem}{#2theorem~#1#3} 
\Crefformat{theorem}{#2Theorem~#1#3} 

\crefname{theoremN}{theorem}{theorems}
\crefformat{theoremN}{#2theorem~#1#3} 
\Crefformat{theoremN}{#2Theorem~#1#3} 

\crefname{enumi}{}{}
\crefformat{enumi}{#2#1#3}
\Crefformat{enumi}{#2#1#3}

\crefname{assumption}{assumption}{Assumptions}
\crefformat{assumption}{#2assumption~#1#3} 
\Crefformat{assumption}{#2Assumption~#1#3} 

\crefname{construction}{construction}{Constructions}
\crefformat{construction}{#2construction~#1#3} 
\Crefformat{construction}{#2Construction~#1#3} 

\crefname{question}{question}{Questions}
\crefformat{question}{#2question~#1#3} 
\Crefformat{question}{#2Question~#1#3} 

\crefname{equation}{}{}
\crefformat{equation}{(#2#1#3)} 
\Crefformat{equation}{(#2#1#3)}

\numberwithin{equation}{section}

\theoremstyle{nonumberplain}
\theoremsymbol{\ensuremath{\blacksquare}}

\newtheorem{proof}{Proof}
\newcommand\pf[1]{\newtheorem{#1}{Proof of \Cref{#1}}}

\newcommand\bA{{\mathbb A}}

\newcommand\bC{{\mathbb C}}
\newcommand\bD{{\mathbb D}}

\newcommand\bG{{\mathbb G}}
\newcommand\bH{{\mathbb H}}

\newcommand\bK{{\mathbb K}}
\newcommand\bL{{\mathbb L}}
\newcommand\bM{{\mathbb M}}
\newcommand\bN{{\mathbb N}}

\newcommand\bP{{\mathbb P}}

\newcommand\bR{{\mathbb R}}
\newcommand\bS{{\mathbb S}}
\newcommand\bT{{\mathbb T}}
\newcommand\bU{{\mathbb U}}

\newcommand\bZ{{\mathbb Z}}

\newcommand\fa{{\mathfrak a}}

\newcommand\fg{{\mathfrak g}}

\newcommand\fl{{\mathfrak l}}

\newcommand\fn{{\mathfrak n}}

\DeclareMathOperator{\Ad}{Ad}
\DeclareMathOperator{\id}{id}

\DeclareMathOperator{\End}{\mathrm{End}}
\DeclareMathOperator{\Hom}{\mathrm{Hom}}

\DeclareMathOperator{\Aut}{\mathrm{Aut}}

\DeclareMathOperator{\im}{\mathrm{im}}

\newcommand{\cat}[1]{\textsc{#1}}

\newcommand\spr[1]{\cite[\href{https://stacks.math.columbia.edu/tag/#1}{Tag {#1}}]{stacks-project}}
\newcommand{\qedhere}{\mbox{}\hfill\ensuremath{\blacksquare}}

\newcommand{\xrightarrowdbl}[2][]{%
  \xrightarrow[#1]{#2}\mathrel{\mkern-14mu}\rightarrow
}

\title{Pervasive ellipticity in locally compact groups}
\author{Alexandru Chirvasitu}

\begin{document}

\date{}

\newcommand{\Addresses}{{
  \bigskip
  \footnotesize

  \textsc{Department of Mathematics, University at Buffalo}
  \par\nopagebreak
  \textsc{Buffalo, NY 14260-2900, USA}  
  \par\nopagebreak
  \textit{E-mail address}: \texttt{achirvas@buffalo.edu}


}}

\maketitle

\begin{abstract}
  A topological group is (openly) almost-elliptic if it contains a(n open) dense subset of elements generating relatively-compact cyclic subgroups. We classify the (openly) almost-elliptic connected locally compact groups as precisely those with compact maximal semisimple quotient and whose maximal compact subgroups act trivial-weight-freely on the Euclidean quotients of the closed derived series' successive layers. In particular, an extension of a compact connected group $\mathbb{K}$ by a connected, simply-connected solvable Lie group $\mathbb{L}$ is (openly) almost-elliptic precisely when the weights of the $\mathbb{K}$-action on $Lie(\mathbb{L})$ afforded by the extension are all non-trivial.
\end{abstract}


\vspace{.5cm}

\noindent{MSC 2020: 22E15; 54D30; 22D05; 22D12; 22E25; 22E60; 22E40; 17B05


}


\section*{Introduction}

The paper is concerned with the structure theory of Lie or, more generally, locally compact groups (always assumed Hausdorff); specifically, we will be interested in phenomena attendant to having fairly large subsets consisting of elements that generate relatively compact subgroups. The following definition collects some of the terms featuring prominently below.

\begin{definition}\label{def:ell}
  \begin{enumerate}[(1),wide]
  \item\label{item:def:ell:ell.el} Following \cite[\S IX.7]{helg_dglgssp}, we refer to an element of a topological group as \emph{elliptic} if it generates a relatively compact subgroup. These are the \emph{compact} elements of \cite[Definition 2.1(b)]{2410.08083v1} (also \cite[Definition 1.14]{MR1002207}), where \emph{elliptic} has a slightly different meaning. 
    
  \item\label{item:def:ell:ell.set} A subset $S\subseteq \bG$ of a topological group is \emph{pointwise-elliptic} if all of its elements are.

  \item\label{item:def:ell:alm.ell.set} $S$ is \emph{almost-elliptic} if it contains a pointwise-elliptic set dense in $S$.

  \item\label{item:def:ell:op.alm.ell.set} And $S$ is \emph{openly almost-elliptic} if it contains a pointwise-elliptic set dense and open in $S$.
  \end{enumerate}
\end{definition}

Most of the material below revolves around characterizing and classifying connected locally compact almost-elliptic groups in the sense of \Cref{def:ell}\Cref{item:def:ell:alm.ell.set}. There are several angles to the original motivation (some aspects of which are by now quite distant from the actual content of the paper).

In part, these considerations arose out of the desire to study (in \cite{Chirvasitu2024}, initially), for a \emph{Banach Lie group} \cite[Definition IV.1]{neeb-inf} $\bG$, the space of closed (Lie, compact, various other constraints) subgroups $\bH\le \bG$ equipped with the \emph{upper Vietoris topology} \cite[\S 1.2]{ct_vietoris}: a system of neighborhoods around $\bH\le \bG$ consists of
\begin{equation*}
  \left\{\text{closed, Lie, etc. }\bH'\le \bG\ :\ \bH'\subseteq U\right\}
  ,\quad
  \text{varying neighborhoods }U=\overset{\circ}{U}\supseteq \bH.
\end{equation*}
This is not the only sensible candidate for a topology that will roughly serve the same purposes: one might, for instance, range only over neighborhoods $U$ of the form $\bH\cdot V$ for symmetric ($V=V^{-1}$) neighborhoods of $1\in \bG$ (as \cite[Definition 3.1]{Chirvasitu2024} does). When sampling only \emph{compact} subgroups of $\bG$ the two obviously coincide, but it is not clear that this is so for for \emph{non}-compact (Lie, say) subgroups $\bH\le \bG$.

In that connection, in attempting to deduce structure constraints on subgroups $\bH'\le \bG$ ``asymptotically very close'' to non-compact $\bH\le \bG$, one might naturally expect that such groups, if non-compact, contain discrete copies of $\bZ$ that are similarly asymptotically close to $\bH$. A dearth of such discrete subgroups, on the other hand, would translate to elements of $\bH'$ having a propensity to generate relatively compact subgroups: something akin to the almost-ellipticity of \Cref{def:ell}\Cref{item:def:ell:alm.ell.set}.

On a different but related note, there are connections to the vast literature on and around the celebrated \emph{Jordan theorem} (\cite[p.91]{MR1581645}, \cite[\S 1, statement I]{zbMATH03223755}, \cite[Theorem 8.29]{ragh}, \cite[Theorem 2]{MR3830471}, etc. etc.): finite subgroups of connected finite-dimensional Lie groups $\bG$ have normal abelian subgroups of index bounded by some constant depending only on $\bG$.

In the process of extending \cite[Theorem]{zbMATH03510594} Jordan's theorem to \emph{torsion} subgroups, Lee shows \cite[Lemma 2]{zbMATH03510594} that torsion subgroups of connected Lie groups are \emph{relatively compact} (i.e. \cite[\S 0.3, pre Theorem 3.9]{ks_diff} have compact closure). Rephrasing slightly awkwardly: if $\Gamma\le \bG$ is a subgroup of a connected Lie group with all cyclic subgroups of $\Gamma$ finite, then $\Gamma$ is relatively compact. Relaxing the finiteness of the cyclic subgroups of $\Gamma$ to relative compactness is at that point a fairly natural move, and one easily conducive to raising the question of what connected Lie groups with dense pointwise-elliptic subsets might look like (sample question: must they be compact?). 

A brief reminder on assorted vocabulary will help summarize some of the main results.
\begin{itemize}[wide]
\item A \emph{pro-torus} \cite[Definitions 9.30]{hm5} is an abelian compact connected group.

\item Given a representation of a topological group $\bK$ on a finite-dimensional complex vector space $V$ (in symbols: $\bG\xrightarrow{\rho}GL(V)$ or $\rho:\bG\circlearrowright V$), a \emph{weight} \cite[Definition II.8.2]{btd_lie_1995} of $\rho$ with respect to a compact abelian $\bA\le \bG$ (sometimes \emph{$\bA$-weight of $\rho$} for short, or variants such as \emph{$\bA$-$\rho$-weight}) is a character
  \begin{equation*}
    \chi\in \widehat{\bA}:=\Hom(\bA,\bS^1)
    \quad\left(\text{\emph{Pontryagin dual} \cite[pp.ix-x]{hm5}}\right)
  \end{equation*}
  appearing as a summand of the restriction $\rho|_{\bA}$.

\item We extend the weight vocabulary to \emph{real} representations $\rho:\bK\circlearrowright V$ by first passing to their respective complexifications
  \begin{equation*}
    \rho_{\bC}:=\rho\otimes \id
    \quad:\quad
    \bK
    \quad
    \circlearrowright
    \quad
    V_{\bC}:=V\otimes_{\bR}\bC.
  \end{equation*}  
\end{itemize}

We denote collectively by $\Ad$ the usual \cite[\S III.3.12]{bourb_lie_1-3} \emph{adjoint} (conjugation) actions of subgroups of $\bG$ on either subgroups thereof or Lie algebras of such, relying on context to specify the various parameters. 

\begin{theoremN}\label{thn:conn.cpct.by.ss.solv}
  Let
  \begin{equation}\label{eq:k.by.l}
    \{1\}
    \to
    \bL
    \lhook\joinrel\xrightarrow{\quad}
    \bG
    \xrightarrowdbl{\quad}
    \bK
    \to
    \{1\}
  \end{equation}
  be an extension of a compact connected group $\bK$ by solvable Lie, connected, simply-connected $\bL$. The following conditions are equivalent. 

  \begin{enumerate}[(a)]

  \item\label{item:thn:conn.cpct.by.ss.solv:kfr.dns} For any (equivalently, all) maximal compact $\bM\le \bG$, the subset
    \begin{equation*}
      \bM_{\Ad,fr}
      :=
      \left\{s\in \bM\ :\ \Ad_s-1\in GL(\fl:=Lie(\bL))\text{ is invertible}\right\}
      \subseteq \bM
    \end{equation*}
    of elements acting freely on $\fl^{\times}:=\fl\setminus \{0\}$ is dense.

  \item\label{item:thn:conn.cpct.by.ss.solv:kfr.clstr} For any (all) $1\in \bM$ is a \emph{cluster point} \cite[Definition 4.9]{wil_top} of $\bM_{\Ad,fr}$.
    
  \item\label{item:thn:conn.cpct.by.ss.solv:tfr.dns} The set $\bT_{\Ad,fr}\subseteq \bT$ is dense for every (equivalently, any one) maximal pro-torus $\bT\le \bM$ of any (every) maximal compact $\bM\le \bG$.

  \item\label{item:thn:conn.cpct.by.ss.solv:tfr.clstr} $1\in \bT$ is a cluster point for $\bT_{\Ad,fr}$ for every (any) maximal pro-torus $\bT\le \bM$.
    
  \item\label{item:thn:conn.cpct.by.ss.solv:ntw} The weights of $\Ad:\bM\circlearrowright \fl$ with respect to every (any) maximal pro-torus $\bT\le \bM$ are all non-trivial. 

  \item\label{item:thn:conn.cpct.by.ss.solv:dns.ell} $\bG$ is almost-elliptic in the sense of \Cref{def:ell}\Cref{item:def:ell:alm.ell.set}.

  \item\label{item:thn:conn.cpct.by.ss.solv:op.dns.ell} $\bG$ is openly almost-elliptic in the sense of \Cref{def:ell}\Cref{item:def:ell:op.alm.ell.set}.
  \end{enumerate}
\end{theoremN}

We eventually supersede that result in \Cref{thn:char.conn.alm.ell.lie}, whose proof however builds on that of the earlier theorem. Some more pertinent notation pertaining to a connected locally compact $\bG$:
\begin{itemize}[wide]
\item $\bG_{sol}\le \bG$ is its \emph{(pro-solvable) radical} \cite[Definition 10.23]{hm_pro-lie-bk}: the largest connected normal subgroup that is \emph{pro-solvable} in the sense \cite[Definition 10.12]{hm_pro-lie-bk} that its Lie quotients are solvable.

\item Following \cite[Definition 10.8]{hm_pro-lie-bk}, we write
  \begin{equation*}
    \bG=
    \bG^{((0))}
    \ge
    \bG^{((1))}
    \ge
    \cdots    
  \end{equation*}
  for the (possibly transfinite) \emph{closed derived series}, defined recursively for successor ordinals $\alpha+1$ by $\bG^{((\alpha+1))}=\overline{[\bG^{((\alpha))},\bG^{((\alpha))}]}$ (with $[-,-]$ denoting, as usual, the subgroup generated by commutators).

  $\bG^{((n))}$ is (for finite $n$, which are the only indices we will in fact need) what \cite[\S III.9.1]{bourb_lie_1-3} denotes by $\overline{D^n}\bG$.

\item We use German characters in the obvious fashion for the \emph{pro-Lie algebras} \cite[Definition 3.6 and Theorem 3.12]{hm_pro-lie-bk} respectively attached to locally compact connected groups: $\fg=Lie(\bG)$, $\fl=Lie(\bL)$, etc.
  
\item For a locally compact connected abelian group $\bG$, by necessity (\cite[Theorem 4.2.2]{de}, \cite[\S 4.13, first Theorem]{mz}) of the form $\bR^d\times \bK$ for compact connected abelian $\bK$, we write $\fg_{nc}$ for the Lie algebra of the quotient $\bR^d\cong \bG/\bK$. 
\end{itemize}


\begin{theoremN}\label{thn:char.conn.alm.ell.lie}
  Let $\bG$ be a connected locally compact group.
  \begin{enumerate}[(1)]
  \item\label{item:thn:char.conn.alm.ell.lie:glob} $\bG$ is almost-elliptic if and only if
    \begin{enumerate}[(a)]
    \item\label{item:thn:char.conn.alm.ell.lie:ss.cpct} the semisimple quotient $\bG/\bG_{sol}$ is compact;

    \item\label{item:thn:char.conn.alm.ell.lie:lyrs.act} and for every (equivalently, any one) maximal compact subgroup $\bK\le \bG$ the representations $\Ad:\bK \circlearrowright \left(\fg^{((n))}_{sol}/\fg^{((n+1))}_{sol}\right)_{nc}$ are trivial-weight-free.
    \end{enumerate}
    
  \item\label{item:thn:char.conn.alm.ell.lie:loc} Equivalently, in \Cref{item:thn:char.conn.alm.ell.lie:lyrs.act} it is enough to consider the respective adjoint action of maximal compact subgroups of $\bG_{((n+1))}:=\bG/\bG^{((n+1))}_{sol}$ on $\left(\fg^{((n))}_{sol}/\fg^{((n+1))}_{sol}\right)_{nc}$.

  \item\label{item:thn:char.conn.alm.ell.lie:auto.opn} Moreover, if almost-elliptic then $\bG$ is in fact openly so. 
  \end{enumerate}
\end{theoremN}

An amalgamation of \Cref{thn:char.conn.alm.ell.lie} and \Cref{th:ext.by.rd.perm}, on which the former relies, provides a permanence result for almost-ellipticity. 

\begin{theoremN}\label{thn:perm.vec.sbgp}
  Let $\bG$ be a connected locally compact group and $\bL\trianglelefteq \bG$ a normal closed vector subgroup.

  $\bG$ is (openly) almost elliptic if and only if $\bG/\bL$ is respectively so, along with all (equivalently, any one) subgroup of $\bG$ of the form $\bL\cdot \bK\cong \bL\rtimes \bK$ for maximal compact $\bK\le \bG$. 
\end{theoremN}

There is also \Cref{th:sel.cj}, here playing an auxiliary role, but perhaps of some independent interest:

\begin{theoremN}\label{thn:cj.cls}
  For any neighborhood $V\ni 1$ in a connected locally compact group $\bG$ and maximal compact $\bK\le \bG$ there is some open dense subset $U_{\bK}\subseteq \bK$ so that
  \begin{equation*}
    \bG \ni
    t
    \text{ sufficiently close to $U_{\bK}$}
    \xRightarrow{\quad}
    \exists\left(s\in V\right)\left(sts^{-1}\in \bK\right). 
  \end{equation*}
\end{theoremN}

The paper concludes with a number of asides and general remarks on maximal compact subgroups, flowing out of the main thread. An example is \Cref{pr:lc.alm.conn.max.cpct.inv}, to the effect that a compact automorphism group of an almost-connected locally compact group $\bG$ leaves some maximal compact subgroup of $\bG$ invariant. 

\subsection*{Acknowledgments}

I am grateful for helpful suggestions from K.-H. Neeb. 


\section{Groups with large elliptic sets}\label{se:gp.lg.ell}

Separation (i.e. being \emph{Hausdorff}, or \emph{$T_2$} \cite[Definition 13.5]{wil_top}) is always assumed for topological groups. The vast majority of the Lie groups of interest below are modeled on finite-dimensional manifolds (as opposed to the \emph{Banach} Lie groups of \cite[\S III.1.1, Definition 1]{bourb_lie_1-3}, \cite[\S VI.5]{lang-fund}, \cite[Definition IV.1]{neeb-inf}, etc.), so finite dimensionality is taken for granted and all departures from it (as far as Lie groups are concerned) will be made explicit.

\begin{remark}\label{re:dense.tors}
  In reference to \cite[Lemma 2]{zbMATH03510594} (recalled in the Introduction), to the effect that torsion subgroups $\Gamma\le \bG$ of connected Lie groups are relatively compact, it is crucial that $\Gamma\le \bG$ be a sub\emph{group} rather than just a sub\emph{set}. It is perfectly possible for a connected, non-compact, Lie group to have dense torsion: \Cref{ex:dns.ell.set}.
\end{remark}

\begin{example}\label{ex:dns.ell.set}
  Consider the semidirect product $\bG:=V\rtimes \bK$ afforded by a representation $\rho:\bK\circlearrowright V$ (which notation refers to a continuous morphism $\bK\xrightarrow{\rho} GL(V)$) of a compact connected Lie group $\bK$ on a f.d. real vector space $V$.

  An element $(v,s)\in V\rtimes \bK$ will be conjugate to one of $\bK$ (and hence elliptic) as soon as $v$ is in the image of $\rho(s)-1\in \End(V)$. It follows that
  \begin{itemize}[wide]
  \item the set of elliptic elements is dense in $\bG$ whenever the trivial weight does not feature in $\rho$.
    
  \item so the same goes for the set of torsion elements: every elliptic element topologically generates a compact Lie group; all such, in turn, are products of tori and finite abelian groups \cite[Corollary 4.2.6]{de}, so have dense torsion. 
  \end{itemize}
  For the more glib assertions we refer to \Cref{th:no.triv.wt}.
\end{example}

\Cref{ex:dns.ell.set} can in fact be turned into a \emph{characterization} result for representations giving rise to semidirect products with ``large'' sets of elliptic elements. 

\begin{theorem}\label{th:no.triv.wt}
  Let
  \begin{equation}\label{eq:k.by.v}
    \{1\}
    \to
    \left(V\cong \bR^d\right)
    \lhook\joinrel\xrightarrow{\quad}
    \bG
    \xrightarrowdbl{\quad}
    \bK
    \to
    \{1\}
  \end{equation}
  be an extension of a compact connected group $\bK$ by a vector group $V\cong \bR^d$ and denote by $\rho:\bK \circlearrowright V$ the underlying $\bK$-representation. The following conditions are equivalent
  
  \begin{enumerate}[(a)]

  \item\label{item:th:no.triv.wt:kfr.dns} The subset
    \begin{equation*}
      \bK_{\rho,fr}
      :=
      \left\{s\in \bK\ :\ \rho(s)-1\in GL(V)\text{ is invertible}\right\}
      \subseteq \bK
    \end{equation*}
    of elements acting freely on $V^{\times}:=V\setminus \{0\}$ is dense.

  \item\label{item:th:no.triv.wt:kfr.clstr} $1\in \bK$ is a \emph{cluster point} \cite[Definition 4.9]{wil_top} of $K_{\rho,fr}$.

    
  \item\label{item:th:no.triv.wt:tfr.dns} The set $\bT_{\rho,fr}\subseteq \bT$ is dense for every (equivalently, any one) maximal pro-torus $\bT\le \bK$.

  \item\label{item:th:no.triv.wt:tfr.clstr} $1\in \bT$ is a cluster point for $\bT_{\rho,fr}$ for every (any) maximal pro-torus $\bT\le \bK$.
    
  \item\label{item:th:no.triv.wt:ntw} The weights of $\rho$ with respect to every (any) maximal pro-torus $\bT\le \bK$ are all non-trivial. 

  \item\label{item:th:no.triv.wt:dns.ell} $\bG$ is almost-elliptic in the sense of \Cref{def:ell}\Cref{item:def:ell:alm.ell.set}.

  \item\label{item:th:no.triv.wt:op.dns.ell} $\bG$ is openly almost-elliptic in the sense of \Cref{def:ell}\Cref{item:def:ell:op.alm.ell.set}.
  \end{enumerate}  
\end{theorem}
\begin{proof}
  The statement's groups $\bG$ are precisely the semidirect products $\bG\cong V\rtimes_{\rho}\bK$, since \Cref{eq:k.by.v} always splits (\cite[Th\'eor\`eme 1]{SSL_1954-1955__1__A24_0}, \cite[Exercise 19 for \S III.9]{bourb_lie_1-3}, etc.)
  
  Observe that all conditions hold vacuously when $V=\{0\}$: everything in sight is elliptic and acts freely on the empty set $V^{\times}$, $\rho$ has no weights at all, etc. We can thus safely henceforth ignore that trivial case and assume $V$ non-zero. Nor is there any harm in assuming $\bK$ Lie whenever convenient, substituting for it the Lie \cite[Theorem I.3.11]{btd_lie_1995} image of $\bK\xrightarrow{\rho}GL(V)$.

  We relegate the purely representation-theoretic mutual equivalence of \Cref{item:th:no.triv.wt:kfr.dns}-\Cref{item:th:no.triv.wt:ntw} to \Cref{pr:rep.content}. To relate those conditions to \Cref{item:th:no.triv.wt:dns.ell} observe first that the subgroup $\bK\le \bG:=V\rtimes_{\rho} \bK$ is maximal compact and conjugate to all such by \cite[\S VII.3, Proposition 3]{bourb_int_7-9_en} (connectedness is irrelevant here), so $(v,s)\in \bG$ is elliptic precisely when conjugate to an element of $\bK\le \bG$. A straightforward computation shows that this translates precisely to
  \begin{equation}\label{eq:v.in.im}
    v\in \im (\rho(s)-1)\le V;
  \end{equation}
  condition \Cref{item:th:no.triv.wt:dns.ell}, then, translates to
  \begin{equation}\label{eq:dns.in.vk}
    \left\{(v,s)\in V\rtimes \bK\ :\ \text{\Cref{eq:v.in.im} holds}\right\}
    \quad
    \overset{\text{dense}}{\subseteq}
    \quad
    V\rtimes \bK.
  \end{equation}
  \begin{enumerate}[label={},wide]
  \item \textbf{\Cref{item:th:no.triv.wt:ntw} $\xRightarrow{\quad}$ \Cref{item:th:no.triv.wt:dns.ell}.} Simply observe that
    \begin{equation*}
      \left\{s\in \bK\ :\ \rho(s)-1\text{ invertible}\right\}
      \subseteq
      \bK
    \end{equation*}
    is dense. For such an $s\in \bK$ \emph{all} $v\in V$ will satisfy \Cref{eq:v.in.im}, hence the conclusion. 
    
    
  \item \textbf{\Cref{item:th:no.triv.wt:dns.ell} $\xRightarrow{\quad}$ \Cref{item:th:no.triv.wt:ntw}.} We will argue for the contrapositive. The assumption that $\rho$ admits the trivial character of $\bT$ as a weight amounts to saying that all elements of $\bT$ (hence also of $\bK=\bigcup_s s\bT s^{-1}$ \cite[Theorem 9.32(ii)]{hm5}) have 1 as an eigenvalue on $V$.

    Consider an element $s\in \bK$ whose 1-eigenspace on $V$ has minimal dimension, yielding an orthogonal decomposition
    \begin{equation}\label{eq:v.decomp}
      V=\ker\left(1-\rho(s)\right)\oplus \im\left(1-\rho(s)\right)
    \end{equation}
    with respect to a $\bK$-invariant inner product on $V$ fixed beforehand. As noted in \cite[Remark post Proposition 13.4]{salt_divalg}, rank is \emph{lower semicontinuous} \cite[Problem 7K]{wil_top} in families of operators. The assumed maximality of the second-summand dimension in \Cref{eq:v.decomp} then implies (by a version of \cite[Proposition 13.4]{salt_divalg}, say) that the trivial $V$-fibered \emph{bundle} \cite[Definition 4.7.2]{hus_fib} over some small neighborhood $\bK\supseteq U\ni s$ decomposes as a direct sum (or \emph{Whitney sum} \cite[Definition 2.4.2]{hus_fib}) 
    \begin{equation}\label{eq:bdl.split.u}
      U\times V\cong E_{\ker}\oplus E_{\im}
      ,\quad
      \text{fibers}
      \left\{
        \begin{aligned}
          E_{\ker,s'}&=\ker\left(1-\rho(s')\right)\\         
          E_{\im,s'}&=\im\left(1-\rho(s')\right)
        \end{aligned}
      \right.
      ,\quad
      s'\in U
    \end{equation}
    of again trivial bundles. For $s'\in U$ the element
    \begin{equation*}
      \im\left(1-\rho\left(s'\right)\right)
      \in
      \text{\emph{Grassmannian} \cite[\S 15.4]{fh_rep-th} of }V
    \end{equation*}    
    will thus be close to $\im\left(1-\rho\left(s\right)\right)$. Consequently, no element
    \begin{equation*}
      (v,s)\in V\rtimes \bK
      ,\quad
      0\ne v\in \ker\left(1-\rho(s)\right)
    \end{equation*}
    can lie in the closure of the set of elliptic elements.

  \item \textbf{\Cref{item:th:no.triv.wt:op.dns.ell} $\xRightarrow{\quad}$ \Cref{item:th:no.triv.wt:dns.ell}} is obvious.  
    
  \item \textbf{\Cref{item:th:no.triv.wt:ntw} $\xRightarrow{\quad}$ \Cref{item:th:no.triv.wt:op.dns.ell}.} If $\rho$ carries only non-trivial weights then $\bK_{\rho,fr}\subseteq \bK$ is in fact \emph{open} dense, and for $s\in \bK_{\rho,fr}$ all $(v,s)\in V\rtimes \bK$ are elliptic.
  \end{enumerate}
\end{proof}

\begin{proposition}\label{pr:rep.content}
  Conditions \Cref{item:th:no.triv.wt:kfr.dns}-\Cref{item:th:no.triv.wt:ntw} of \Cref{th:no.triv.wt} are mutually equivalent for any finite-dimensional representation $\rho:\bK\circlearrowright V$ of a compact connected group $\bK$. 
\end{proposition}
\begin{proof}
  Maximal pro-tori are conjugate in $\bK$ \cite[Theorem 9.32(i)]{hm5}, so any of the statements above valid for one will indeed be valid for all; the split-quantifier sentences can thus indeed be treated as single units. 

  The mutual equivalence of conditions \Cref{item:th:no.triv.wt:kfr.dns,item:th:no.triv.wt:kfr.clstr,item:th:no.triv.wt:tfr.dns,item:th:no.triv.wt:tfr.clstr} is summarized in
  
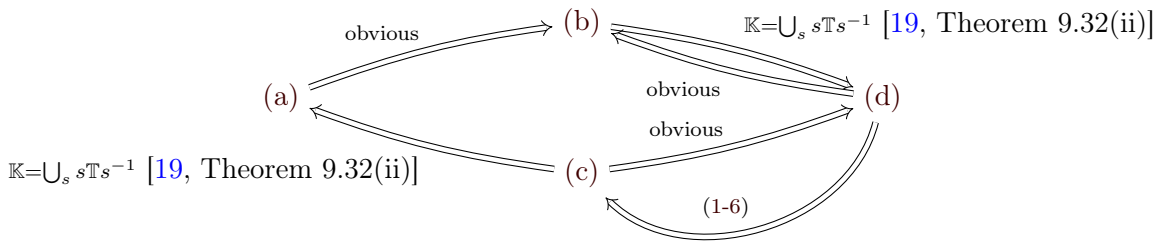
\begin{equation*}
    \begin{tikzpicture}[>=stealth,auto,baseline=(current  bounding  box.center)]
      \path[anchor=base] 
      (0,0) node (l) {\Cref{item:th:no.triv.wt:kfr.dns}}
      +(4,1) node (u) {\Cref{item:th:no.triv.wt:kfr.clstr}}
      +(4,-1) node (d) {\Cref{item:th:no.triv.wt:tfr.dns}}
      +(8,0) node (r) {\Cref{item:th:no.triv.wt:tfr.clstr}}
      ;
      \draw[-implies,double equal sign distance] (l) to[bend left=6] node[pos=.5,auto] {$\scriptstyle \text{obvious}$} (u);
      \draw[-implies,double equal sign distance] (u) to[bend left=6] node[pos=.5,auto] {$\scriptstyle \bK=\bigcup_s s\bT s^{-1}$ \cite[Theorem 9.32(ii)]{hm5}} (r);
      \draw[-implies,double equal sign distance] (r) to[bend left=6] node[pos=.5,auto] {$\scriptstyle \text{obvious}$} (u);
      \draw[implies-,double equal sign distance] (l) to[bend right=6] node[pos=.5,auto,swap] {$\scriptstyle  \bK=\bigcup_s s\bT s^{-1}$ \cite[Theorem 9.32(ii)]{hm5}} (d);
      \draw[-implies,double equal sign distance] (d) to[bend right=6] node[pos=.5,auto] {$\scriptstyle \text{obvious}$} (r);
      \draw[-implies,double equal sign distance] (r) to[bend left=60] node[pos=.5,auto,swap] {$\scriptstyle \text{\Cref{eq:sml.ne1.eigs}}$} (d);
    \end{tikzpicture}
  \end{equation*}
  given
  \begin{equation}\label{eq:sml.ne1.eigs}
    \forall\left(T\in U(n)\right)
    \exists\left(\varepsilon>0\right)
    \quad:\quad
    \left.
      \begin{aligned}
        \|S&-1\|<\varepsilon\\
        ST&=TS\\
        S-1&\text{ invertible}
      \end{aligned}
    \right\}
    \xRightarrow{\quad}
    ST-1\text{ invertible}.
  \end{equation}
  The equivalence \Cref{item:th:no.triv.wt:tfr.clstr} $\iff$ \Cref{item:th:no.triv.wt:ntw} is also self-evident: the elements of $\bT$ where any one of the finitely many non-trivial $\bT$-weights of $\rho$ is non-trivial constitute an open dense subset, so their intersection is dense as well.
\end{proof}

Dropping the connectedness assumption on $\bK$ in \Cref{th:no.triv.wt} will render the implication \Cref{item:th:no.triv.wt:ntw} $\Rightarrow$ \Cref{item:th:no.triv.wt:dns.ell} generally invalid.

\begin{example}\label{ex:z2.act.inv}
  Take for $\bK$ a semidirect product $\bT\rtimes \left(\bZ/2\right)$ for a torus $\bT$ acted upon by $\bZ/2$ by inversion: the generator $\sigma\in \bZ/2$ interchanges $z^{\pm 1}\in \bT$.

  One can easily arrange for $\rho|_{\bT}$ to have no trivial summands (so that \Cref{item:th:no.triv.wt:ntw} holds). $\rho(\sigma)\in GL(V)$ is a non-scalar reflection, and $\rho\left(\bT\sigma\right)$ consists entirely of conjugates thereof. It follows that no
  \begin{equation*}
    (v,\sigma)
    ,\quad
    0\ne v\in V
    ,\quad
    \rho(\sigma)v=v
  \end{equation*}
  can be approximated arbitrarily by elements $(v,s)$ satisfying \Cref{eq:v.in.im}.
\end{example}

\begin{remark}\label{re:g0.not.enough}
  \Cref{ex:z2.act.inv} also shows that the property of having a dense set of elliptic elements, for a Lie group with finitely many components, cannot be checked only on the identity component. Or again: said property is not invariant under finite extensions.
\end{remark}

In part, the proof of \Cref{th:no.triv.wt} does go through for disconnected $\bK$; we isolate the statement for future use. 

\begin{lemma}\label{le:alm.surj.cbd}
  Let \Cref{eq:k.by.l} be an extension of a compact group $\bK$ by solvable Lie, connected, simply-connected $\bL$.

  \begin{enumerate}[(1)]
  \item\label{item:le:alm.surj.cbd:set} For any maximal compact $\bM\le \bG$, the elliptic elements in $\bG$ are 
    \begin{equation}\label{eq:vxx}
      \left\{(v,s)\in \bL\rtimes \bM\cong \bL\cdot \bM\subseteq \bG\ :\ \exists \left(x\in \bL\right)\left(x^{-1}\cdot \Ad_s x=v\right)\right\}.
    \end{equation}

  \item\label{item:le:alm.surj.cbd:dns} In particular, $\bG$ is almost-elliptic if and only if for any (equivalently, all) maximal compact $\bM\le \bG$ the set \Cref{eq:vxx} is dense in $\bG$.
  \end{enumerate}
\end{lemma}
\begin{proof}
  The extension \Cref{eq:k.by.l} splits \cite[Exercise 19 for \S III.9]{bourb_lie_1-3} as $\bG\cong \bL\rtimes \bK$, so the maximal compact subgroups of $\bG$ are the conjugates \cite[Theorems A, B and C]{zbMATH00540631} of any such semidirect factor $\bK$. The conclusion thus follows by the argument casting condition \Cref{item:th:no.triv.wt:dns.ell} of \Cref{th:no.triv.wt} as \Cref{eq:dns.in.vk}. 
\end{proof}

\Cref{re:g0.not.enough} notwithstanding, there are some simple permanence results for almost-ellipticity which we record for whatever use they possess; we omit the proofs and henceforth take them for granted.

\begin{lemma}\label{le:alm.ell.perm}
  \begin{enumerate}[(1)]
  \item\label{item:le:alm.ell.perm:quot.opn} For any topological group $\bG$, quotients and open subgroups are (openly) almost-elliptic as soon as $\bG$ is respectively so.

  \item\label{item:le:alm.ell.perm:ell.in.quot} An element of $\bG$ is elliptic if and only if its image in any (some) quotient $\bG/\bK$ by a compact normal subgroup is elliptic. 
    
  \item\label{item:le:alm.ell.perm:quot.by.cpct} In particular, if $\bK\trianglelefteq \bG$ is compact and normal then $\bG$ is (openly) almost-elliptic if and only if $\bG/\bK$ is respectively so.  \qedhere
  \end{enumerate}
\end{lemma}

We record the following immediate consequence, reducing almost-ellipticity for sufficiently well-behaved locally compact groups to the Lie-group setup. We remind the reader that a topological group $\bG$ is \emph{almost-connected} (e.g. \cite[Preface, p.xvii, Definition 34]{hm5}) is the quotient $\bG/\bG_0$ by its identity connected component is compact.

\begin{corollary}\label{cor:lie.enough}
  An almost-connected locally compact group is (openly) almost elliptic if and only if any one of the following mutually-equivalent conditions holds respectively. 
  \begin{itemize}[wide]
  \item It is a \emph{cofiltered limit} \spr{04AY} of (openly) almost-elliptic Lie groups.

  \item Every compact-kernel Lie quotient of $\bG$ is (openly) almost-elliptic. 

  \item At least one compact-kernel Lie quotient of $\bG$ is (openly) almost-elliptic. 
  \end{itemize}
\end{corollary}
\begin{proof}
  This follows from \Cref{le:alm.ell.perm}\Cref{item:le:alm.ell.perm:quot.by.cpct}, given that almost-connected locally compact groups are the cofiltered limits of their Lie quotients with compact kernels (\cite[\S 4.6, Theorem]{mz} and \cite[Theorem 3.39]{hm_pro-lie-bk}, say). 
\end{proof}

Another sufficient condition for almost-ellipticity inheritance is as follows.

\begin{proposition}\label{pr:gmodl2g}
  Let $\bG$ be an almost-connected locally compact group and $\bL\trianglelefteq \bG$ normal closed.

  $\bG$ is almost elliptic provided $\bG/\bL$ is, along with all (equivalently, any one) subgroup of $\bG$ of the form
    \begin{equation}\label{eq:preim.max.cpct}
      \pi^{-1}\left(\bK\right)
      ,\quad
      \bG\xrightarrowdbl{\quad\pi\quad}\bG/\bL
      ,\quad
      \bK\le \bG/\bL\text{ maximal compact}.
    \end{equation}
\end{proposition}
\begin{proof}
  Almost-connected locally compact groups have maximal compact subgroups, all mutually conjugate, and automatically connected (\cite[\S 4.13, first Theorem]{mz}, \cite[Theorems A, B and C]{zbMATH00540631}. etc.); it follows that `any' and `every' are indeed interchangeable in \Cref{eq:preim.max.cpct}.

  The image $\overline{x}\in \bG/\bL$ of any $x\in \bG$ is by assumption arbitrarily approximable by elements belonging to compact $\bK\le \bG/\bL$, which subgroups may as well be assumed maximal. This means that $x$ is arbitrarily close to elements belonging to subgroups $\pi^{-1}\left(\bK\right)$; in turn, those elements are close to compact subgroups of $\pi^{-1}(\bK)$ (hence also of $\bG$).
\end{proof}

In the much same spirit, the implication of \Cref{pr:gmodl2g} can occasionally be reversed.

\begin{theorem}\label{th:ext.by.rd.perm}
  Let $\bG$ be a connected locally compact group and $\bL\trianglelefteq \bG$ a normal closed vector subgroup.

  $\bG$ is almost elliptic if and only if $\bG/\bL$ is, along with all (equivalently, any one) subgroup of $\bG$ of the form \Cref{eq:preim.max.cpct}. 
\end{theorem}
\begin{proof}
  It will be harmless (\Cref{cor:lie.enough}) to assume $\bG$ Lie. The backward implication is covered by \Cref{pr:gmodl2g}, and by \Cref{le:alm.ell.perm}\Cref{item:le:alm.ell.perm:quot.by.cpct} the quotient $\bG/\bL$ certainly is almost-elliptic whenever $\bG$ is; the more substantive claim is the almost-ellipticity of all (equivalently, at least one, by mutual conjugacy) $\pi^{-1}(\bK)$, assuming that of $\bG$.

  We have a splitting
 \begin{equation*}
   \pi^{-1}(\bK)
   \cong
   \bL\rtimes \bK
   \quad
   \left(\text{\cite[Th\'eor\`eme 1]{SSL_1954-1955__1__A24_0}, \cite[Exercise 19 for \S III.9]{bourb_lie_1-3}, etc.}\right),
 \end{equation*}
 so \Cref{th:no.triv.wt} applies: the claim is that $\bK$ operates on $\bL\cong \bR^d$ with no trivial weights. Assume otherwise, and (as in the proof of \Cref{th:no.triv.wt}) consider an element
 \begin{equation*}
   (v,s)\in \pi^{-1}(\bK)\cong \bL\rtimes \bK
   ,\quad
   v\ne 0
   ,\quad
   sv=v
 \end{equation*}
 with $s\in \bK$ having minimal-dimensional 1-eigenspace and a convergent \emph{net} \cite[Definition 11.2]{wil_top}
 \begin{equation*}
   \bG\ni
   \text{ elliptic }
   t_{\alpha}
   \xrightarrow[\quad\alpha\quad]{\quad}
   (v,s)
   \in \pi^{-1}(\bK)\le \bG.
 \end{equation*}
 The quotient $\bG/\bL$ operates on $\bL$ (for the latter is abelian), and we conflate the elements $\pi(t)$, $t\in \bG$ with the operators they induce on $\bL$. Moreover, by \Cref{th:sel.cj}\Cref{item:th:sel.cj:prx.el} below, $t_{\alpha}$ can be assumed respectively conjugate by $u_{\alpha}\xrightarrow[\alpha]{}1$ to elements in $\pi^{-1}(\bK)\cong \bL\rtimes \bK$; we thus assume $t_{\alpha}\in \pi^{-1}(\bK)$.
 
 The elliptic $\pi(t_{\alpha})$ are all semisimple and hence
 \begin{equation}\label{eq:im.t}
   \im\left(1-\pi(t_{\alpha})\right)|_{\bL_{\bC}:=\bL\otimes_{\bR}\bC}
   =
   \bigoplus_{\lambda\ne 1} \ker\left(\lambda-\pi(t_{\alpha})\right). 
 \end{equation}
 The convergence $\pi(t_{\alpha})\xrightarrow[\alpha]{}s$ entails \cite[Theorem 3]{Newburgh} spectrum convergence, so for the $\lambda\ne 1$ of \Cref{eq:im.t} we have
 \begin{equation}\label{eq:ll0}
   \forall\left(\varepsilon>0\right)
   \ 
   \exists\left(1\ne \lambda_0\in\text{spectrum of }s\right)
   \quad:\quad
   |\lambda-\lambda_0|<\varepsilon
   \text{ for large $\alpha$}.
 \end{equation}
 The assumed ellipticity of $t_{\alpha}$ implies
 \begin{equation*}
   v_{\alpha}\overset{\text{cf. \Cref{eq:v.in.im}}}{\in} \im\left(1-t_{\alpha}\right)
   \xRightarrow[\quad]{\quad\text{\Cref{eq:ll0}}\quad}
   \pi(t_{\alpha})v_{\alpha}-v_{\alpha}
   \text{ bounded away from 0}.
 \end{equation*}
 This contradicts the convergence $t_{\alpha}\xrightarrow[\alpha]{}(v,s)$, given that $sv=v$.
\end{proof}

An \emph{almost-containment} $V\succeq U$ in a topological space means that $V$ contains an open dense subset of $U$. Denoting interiors by `$\circ$' decorations, $V=\overset{\circ}{V}\succeq U$ then says that $V$ is an open \emph{almost-neighborhood} of $U$.  

\begin{theorem}\label{th:sel.cj}
  Let $\bK\le \bG$ be a maximal compact subgroup of a connected locally compact group. 
  
  \begin{enumerate}[(1)]

  \item\label{item:th:sel.cj:prx.gp} We have
    \begin{equation*}
      \forall\left(V=\overset{\circ}{V}\supseteq \bK\right)
      \exists\left(U=\overset{\circ}{U}\succeq \bK\right)
      \forall\left(\text{elliptic }t\in U\right)
      \quad:\quad
      \Braket{t}\subseteq V,
    \end{equation*}
    with $\Braket{\bullet}$ denoting generated subgroups. 
    
  \item\label{item:th:sel.cj:prx.el} There is a function
    \begin{equation*}
      \left(U=\overset{\circ}{U}\succeq \bK\right)
      \xrightarrow[\quad\text{continuous at 1}\quad]{\quad u_{\bullet}\quad}
      \bG
      \quad\text{with}\quad
      \left\{
        \begin{aligned}
          u_1
          &=1\\
          \Ad_{u_t} t
          &\in \bK
            ,\quad
            \forall\text{ elliptic }t\in U.
        \end{aligned}
      \right.
    \end{equation*}
  \end{enumerate} 
\end{theorem}
\begin{proof}
  Recall the usual approximation argument \cite[\S 4.6, Theorem]{mz}: $\bG$ has arbitrarily small normal compact subgroups with Lie quotients. This allows the standing assumption that $\bG$ is Lie.

  By \cite[Theorem 1 and Corollary]{MR6545} compact subgroups contained in sufficiently small neighborhoods of $\bK$ are conjugate to subgroups of $\bK$ by elements close to $1\in \bG$ (see also \cite[Theorem 3.4]{Chirvasitu2024} for the \emph{Banach}-Lie-group version of that result). \Cref{item:th:sel.cj:prx.el}, consequently, reduces to \Cref{item:th:sel.cj:prx.gp}. The rest of the proof focuses on the latter: the claim is that elliptic elements in judiciously-chosen almost-neighborhoods of $\bK\le \bG$ generate subgroups arbitrarily close to $\bK$.

  
  \begin{enumerate}[(I),wide]
  \item \textbf{: reduction to the linear case.} Recall \cite[Definition 5.32]{hm5} that Lie groups are \emph{linear} when they embed into $\mathrm{GL}(n,\bC)$; under the present (almost-)connected assumption the embedding can always be assumed closed (e.g. \cite[Proposition 5]{0801.4234v1}).

    One can always substitute for $\bG$ the image $\Ad(\bG)$ of its adjoint representation
    \begin{equation*}
      \bG
      \xrightarrow{\quad\Ad\quad}
      \mathrm{GL}(\fg)
      ,\quad
      \fg:=Lie(\bG).
    \end{equation*}
    Indeed, if the conclusion (\Cref{item:th:sel.cj:prx.gp} and hence also \Cref{item:th:sel.cj:prx.el}, as noted) is known to hold for $\Ad(\bG)$, then elliptic $t\in \bG$ sufficiently close to $\bK$ can be conjugated into $\bK\cdot \ker\Ad$ by elements close to $1\in \bG$. The conclusion follows from the fact that for \emph{central} closed $\bH\le \bG$ the identity component $\left(\bH\cdot \bK\right)_0$ of $\bH\cdot \bK$ contains $\bK$ as its \emph{unique} maximal compact subgroup.
      
  \item \textbf{: the linear case.} We can now assume $\bG$ to be a closed subgroup of some $\mathrm{GL}(\bC^n)$. The desired open sets $U$ will range over neighborhoods of
    \begin{equation*}
      \bK_{gen}
      :=
      \left\{s\in \bK\ :\ s\text{ \emph{generic} with respect to the embedding }\bK\le \mathrm{GL}(n,\bC)\right\}
      \overset{\text{open dense}}{\subseteq}
      \bK:
    \end{equation*}
    the weights of the underlying $\bK$-representation on $\bC^n$ take distinct values on $s$ (no matter how one realizes $s$ as an element of a maximal torus $\bT\le \bK$). Equivalently, the genericity condition says simply that the \emph{isotypic components} \cite[post Proposition II.1.14]{btd_lie_1995} of $\bC^n$ with respect to $\overline{\Braket{s}}$ coincide with those of $\bT\ni s$ for any maximal torus $\bT\le \bK$.

    For generic $s\in \bK_{gen}$ the ambient space $\bC^n$ decomposes as
    \begin{equation}\label{eq:c.wt.sp}
      \bC^n=\bigoplus_{i=1}^d \bC^n_{s,i}
      ,\quad
      \bC^n_{s,i}\text{ the weight spaces of }\Braket{s},
    \end{equation}
    with genericity ensuring this is also the weight-space decomposition for any maximal torus $\bK\ge \bT$ containing $s$. Such a torus can be identified with $\bT^d\cong\left(\bS^1\right)^d$, each circle factor respectively acting as scalars on one of the $\bC^n_{s,i}$ summands. 

    Elliptic elements $t\in \bG$ sufficiently close to $s$ will again be generic (relative to maximal compact subgroups containing them), and will induce analogous decompositions $\bC^n=\bigoplus_{i=1}^d \bC^n_{t,i}$ (same $d$). Closeness $t\simeq s$ means closeness
    \begin{equation*}
      \forall\left(1\le i\le d\right)
      \quad:\quad
      \bC^n_{t,i}\simeq \bC^n_{s,i}
    \end{equation*}
    in the respective Grassmannians of $\bC^n$, hence $T\simeq 1\in \mathrm{GL}(n,\bC)$ conjugating $t$ inside the torus $\bT^d\cong \bT$ identified above, acting diagonally with respect to \Cref{eq:c.wt.sp}.
  \end{enumerate}
\end{proof}

\begin{remark}\label{re:alm.nbhd.nec}
  The various caveats in the statement of \Cref{th:sel.cj} are certainly not all redundant. The maximality of $\bK$, for instance, does play a role: the statement does not hold for embeddings
  \begin{equation*}
    \bK:=\bS^1\le \bT^2:=\left(\bS^1\right)^2=:\bG,
  \end{equation*}
  given that the set
  \begin{equation*}
    \left\{t\in \bT^2\ :\ \overline{\Braket{t}}=\bT^2\right\}
    \subseteq
    \bT^2
  \end{equation*}
  of \emph{topological generators} is dense \cite[Theorem IV.2.11(ii)]{btd_lie_1995}, so every $s\in \bS^1$ is arbitrarily approximable by (elliptic) elements topologically generating groups $\overline{\Braket{t}}$ \emph{not} close to $\bS^1$ in the \emph{upper Vietoris topology} \cite[\S 1.2]{ct_vietoris}.

  On the other hand, \Cref{ex:alm.nbhd.nec} below shows that one cannot generally strengthen the almost-containment $U\succeq \bK$ of \Cref{th:sel.cj}\Cref{item:th:sel.cj:prx.gp} to actual containment. 
\end{remark}

\begin{example}\label{ex:alm.nbhd.nec}
  Set
  \begin{equation*}
    \bK
    :=
    \text{unitary group }\bU(n)
    \le
    \bG
    :=
    \mathrm{GL}(n,\bC).
  \end{equation*}
  Consider decompositions $\bC^n=\bC^{n-1}\oplus \ell$ for lines $\ell$ approaching $\bC^{n-1}$ in he sense that
  \begin{equation*}
    \ell
    \xrightarrow{\quad\text{upper Vietoris}\quad}
    \bP \bC^{n-1}
    \subset
    \bP \bC^n
    ,\quad
    \bP\bullet:=\text{projective space of lines in $\bullet$}. 
  \end{equation*}
  The operators $t_{\ell,\lambda}$ with
  \begin{equation*}
    \begin{aligned}
      \ker\left(1-t_{\ell,\lambda}\right)
      &=
        \bC^{n-1}\\
      \ker\left(\lambda-t_{\ell,\lambda}\right)
      &=
        \ell
    \end{aligned}
    ,\quad
    \lambda\in \bS^1
  \end{equation*}
  are all elliptic, and approach $1\in \bK$ if $\lambda\to 1$ sufficiently fast. On the other hand, if $\lambda\ne 1$ we have $|\lambda^k-1|\ge \sqrt 3$ for some $k$, so that
  \begin{equation*}
    \lim_{\ell}\sup_{k}\left\|t^k_{\lambda,\ell}-1\right\|
    =
    \lim_{\ell}\sup_{k}\left\|t_{\lambda^k,\ell}-1\right\|
    =\infty.
  \end{equation*}
  In particular, the groups $\Braket{t_{\ell,\lambda}}$ will never be contained in a sufficiently small neighborhood of $\bK=\bU(n)$. 
\end{example}

\pf{thn:conn.cpct.by.ss.solv}
\begin{thn:conn.cpct.by.ss.solv}
  As observed in the proof of \Cref{le:alm.surj.cbd}, \Cref{eq:k.by.l} splits and the resulting non-normal semidirect factors are all maximal compact and mutually conjugate, so we can fix a decomposition $\bG\cong \bL\rtimes \bK$ once and for all and identify all of the statement's $\bM$ with the single $\bK$. 

  Proposition \Cref{pr:rep.content} handles the mutual equivalence of conditions \Cref{item:thn:conn.cpct.by.ss.solv:kfr.dns}-\Cref{item:thn:conn.cpct.by.ss.solv:ntw}; denoting that cluster of claims by `$\left(\cat{rep}\right)$', we are left having to prove the unmarked implications in
  \begin{equation*}
    \begin{tikzpicture}[>=stealth,auto,baseline=(current  bounding  box.center)]
      \path[anchor=base] 
      (0,0) node (l) {$\left(\cat{rep}\right)$}
      +(2,.7) node (u) {\Cref{item:thn:conn.cpct.by.ss.solv:op.dns.ell}}
      +(2,-.7) node (d) {\Cref{item:thn:conn.cpct.by.ss.solv:dns.ell}}
      ;
      \draw[-implies,double equal sign distance] (l) to[bend left=20] node[pos=.5,auto] {$\scriptstyle $} (u);
      \draw[-implies,double equal sign distance] (u) to[bend left=20] node[pos=.5,auto] {$\scriptstyle \text{obvious}$} (d);
      \draw[-implies,double equal sign distance] (d) to[bend left=20] node[pos=.5,auto,swap] {$\scriptstyle $} (l);
    \end{tikzpicture}
  \end{equation*}
  
  \begin{enumerate}[label={},wide]
  \item \textbf{$\left(\cat{rep}\right)$ $\xRightarrow{\quad}$ \Cref{item:thn:conn.cpct.by.ss.solv:op.dns.ell}.} \Cref{le:all.cbdr} shows that
  \begin{equation*}
    \left(\forall s\in \bK\right)
    \left(
      1-\Ad_s\text{ invertible}
      \xRightarrow{\quad}
      \left\{x^{-1}\cdot \Ad_s x\ :\ x\in \bL\right\}
      =
      \bL
    \right).
  \end{equation*}
  The conclusion then follows from \Cref{le:alm.surj.cbd}\Cref{item:le:alm.surj.cbd:set}, given that the set of $s\in \bK$ with invertible $1-\Ad_s$ is (assuming $\left(\cat{rep}\right)$) open dense. 

  \item \textbf{\Cref{item:thn:conn.cpct.by.ss.solv:dns.ell} $\xRightarrow{\quad}$ $\left(\cat{rep}\right)$.} We argue by induction on the length of the closed derived series $\left(\bL^{((n))}\right)_{n}$ (consisting of \emph{characteristic} \cite[p.104, Definition preceding Lemma 5.20]{rot-gp}, hence $\bK$-invariant closed normal subgroups of $\bL$).

    The base case of the induction is that of abelian $\bL$, taken care of by \Cref{th:no.triv.wt}. Denote next by $\bA\trianglelefteq \bL$ the bottom layer $\bA:=\bL^{((\mathrm{max}))}$ of the derived series. The almost-ellipticity of $\bG/\bA$ (\Cref{le:alm.ell.perm}\Cref{item:le:alm.ell.perm:quot.opn}) and the induction hypothesis imply $\left(\cat{rep}\right)$ for $\Ad:\bK\circlearrowright \fl/\fa$, with $\fl:=Lie(\bL)$ and $\fa:=Lie(\bA)$. 
    
    Call $s\in \bK$ as \emph{non-degenerate (on $\fl$)} if $\dim\ker\left(1-\Ad_s|_{\fl}\right)$ is minimal among such dimensions when ranging over $\bK$. The induction hypothesis further allows the assumption that
    \begin{equation*}
      \ker\left(1-\Ad_s|_{\fl}\right)
      =
      \ker\left(1-\Ad_s|_{\fa}\right),
    \end{equation*}
    so the goal is to show that in fact non-degenerate $s\in \bK$ have no 1-eigenvectors in $\fa$ (they have none in $\fl/\fa$ in any case). The set $\bK_{ndg}\subseteq \bK$ of non-degenerate elements is open, and over $\bK_{ndg}$ the trivial bundle with fiber $\bL$ decomposes as a \emph{fiber product} \cite[Definition 2.4.2]{hus_fib}
    \begin{equation*}
      \bK_{ndg}\times \bL
      \cong
      E_{\mathrm{fix}}
      \times_{\bK_{ndg}}
      E_{\mathrm{ess}}
      ,\quad
      \text{fibers}
      \left\{
        \begin{aligned}
          E_{\mathrm{fix},s'}&=\bL^{s'}:=\left\{g\in \bL\ :\ \Ad_{s'}g=g\right\}\\         
          E_{\mathrm{ess},s'}&=\im\left(g\xmapsto{\quad} g^{-1}\cdot \Ad_{s'}g\right)
        \end{aligned}
      \right.
      ,\quad
      s'\in \bK_{ndg}
    \end{equation*}
    (non-linear analogue of \Cref{eq:bdl.split.u}).

    It follows that $(g,s)\in \bL\rtimes \bK$ with non-trivial $g\in \bA\le \bL$ fixed by $s\in \bK_{ndg}$ cannot be approximated arbitrarily by elements
    \begin{equation*}
      (g',s')\in \bL\rtimes \bK
      ,\quad
      g'\in E_{\mathrm{ell},s'},
    \end{equation*}
    yielding the conclusion: by \Cref{le:alm.surj.cbd}\Cref{item:le:alm.surj.cbd:set}, said elements are precisely the elliptic ones.  \qedhere
  \end{enumerate}
\end{thn:conn.cpct.by.ss.solv}

\begin{lemma}\label{le:all.cbdr}
  Let $\bL$ be a connected, simply-connected solvable Lie group and $\varphi\in \Aut(\bL)$ with $1-\varphi$ invertible on $\fl:=Lie(\bL)$.

  The map
  \begin{equation*}
    \bL\ni s
    \xmapsto{\quad\delta=\delta_{\varphi}\quad}
    s^{-1}\cdot \varphi(s)
    \in \bL
  \end{equation*}
  is onto. 
\end{lemma}
\begin{proof}
  Identifying tangent spaces
  \begin{equation*}
    T_1 \bL
    \xrightarrow[\quad\cong\quad]{\quad\text{right translation }R_{g*}\quad}
    T_g \bL
    ,\quad
    g\in \bL,
  \end{equation*}
  the differential of $\delta$ at a generic point $s\in \bL$ is
  \begin{equation*}
    T_s\bL\ni
    v
    \xmapsto{\quad D\delta \quad}
    R_{\varphi(s)*}\Ad_{s^{-1}}\left(\varphi(v)-v\right)
    \in
    T_{\varphi(s)}\bL.
  \end{equation*}
  This is by assumption bijective, so $\delta$ is at any rate a local diffeomorphism \cite[Theorem 4.5]{lee2013introduction}.  
  
  The rest of the argument is inductive on the length of the closed derived series
  \begin{equation*}
    \bL=\bL^{((0))}
    \ge
    \bL^{((1))}
    \ge
    \cdots
    \ge
    \bA:=\bL^{((\mathrm{max}))}
    \gneq
    \{1\},
  \end{equation*}
  the base case (of abelian $\bL$) being clear. The induction hypothesis now delivers the conclusion for $\bL/\bA$, so that the composition
  \begin{equation*}
    \bL
    \xrightarrow{\quad\delta_{\varphi}\quad}
    \bL
    \xrightarrow{\quad\pi\quad}
    \bL/\bA
  \end{equation*}
  is onto. Observe also that
  \begin{equation*}
    \forall \left(a\in \bA\right)
    \forall\left(s\in \bL\right)
    \quad:\quad
    \delta(as)=\Ad_{s^{-1}}\left(\delta a\right)\cdot \delta s,
  \end{equation*}
  so $\im\delta\subset \bL$ is an open submanifold, surjecting onto $\bL/\bA$, and invariant under left translation by $\bA$. This suffices to conclude that $\im\delta=\bL$.
\end{proof}

The simple remark in \Cref{pr:ss.ell.cpct} shows that in a sense the \emph{radical} (\cite[\S III.9.7, Definition 1]{bourb_lie_1-3} for Lie groups, \cite[Definition 10.23]{hm_pro-lie-bk} for connected locally compact groups), containing $V\le V\rtimes \bK$ in \Cref{ex:dns.ell.set}, is unavoidable in trying to produce non-compact instances of almost ellipticity.

\emph{Semisimple} locally compact groups are as in \cite[Definition 10.27]{hm_pro-lie-bk} (\cite[\S I.15, discussion post Theorem 1.125]{Kna02} for Lie groups): connected with trivial radical; see also \cite[Theorem 10.29]{hm_pro-lie-bk} for multiple equivalent characterizations.

\begin{proposition}\label{pr:ss.ell.cpct}
  An almost-connected locally compact group with semisimple identity component is almost-elliptic if and only if it is compact.
\end{proposition}
\begin{proof}
  \Cref{cor:lie.enough} reduces the problem to Lie groups, on which we henceforth focus. Only the $\left(\Rightarrow\right)$ implication requires any work. Note furthermore that
  \begin{itemize}[wide]
  \item almost-ellipticity passes over to quotients and hence requires finitely many components;

  \item and is inherited by open subgroups. 
  \end{itemize}
  It thus suffices, upon substituting the identity component $\bG_0$ for $\bG$, to assume the group connected. 
  
    Observe that the last two components $\bA$ and $\bN$ in the \emph{global Iwasawa decomposition} $\bG=\bK\bA\bN$ \cite[Theorem 6.46]{Kna02} must be trivial whenever $\bG$ has a dense elliptic \emph{set} (group or not):
  \begin{itemize}[wide]
  \item The Lie algebra $\fa$ of the (simply-connected) closed subgroup $\bA$ operates on $\fg:=Lie(\bG)$ with positive real eigenvalues \cite[Lemma 6.45]{Kna02}, so elements sufficiently close to non-trivial elements of $\bA$ cannot be elliptic.

  \item And $\bN$ is trivial because $\fn=[\fa,\fn]$ by \cite[Proposition 6.43]{Kna02} (where $\fn:=Lie(\bN)$).
  \end{itemize}
  This suffices to conclude: compact semisimple Lie groups have finite fundamental groups \cite[Table 1]{MR1039850} and hence compact \emph{universal covers} \cite[Theorem A2.20]{hm5}. 
\end{proof}



\pf{thn:char.conn.alm.ell.lie}
\begin{thn:char.conn.alm.ell.lie}
  As customary by now, we assume $\bG$ Lie: this will make no difference to the conclusion, by \Cref{cor:lie.enough}. The mutual conjugacy of maximal compact subgroups \cite[Theorems A, B]{zbMATH00540631} again justifies the quantifier flexibility, so it suffices to fix a maximal compact $\bK\le \bG$ throughout. \Cref{pr:ss.ell.cpct} moreover reduces the problem to considering only compact $\bG/\bG_{sol}$. 
  
  Note also that maximal compact subgroups $\bK\le \bG/\bG^{((n-1))}_{sol}$ always lift to maximal compact $\widetilde{\bK}\le \bG/\bG^{((n))}_{sol}$: there is \cite[Corollary X.1]{hm_split} a $\bK$-invariant decomposition
  \begin{equation*}
    \text{abelian connected }
    \bG^{((n-1))}_{sol}/\bG^{((n))}_{sol}
    \cong
    \bR^d\times \left(\text{torus }\bT^e\right),
  \end{equation*}
  and the quotient by the $\bR^d$ factor splits \cite[\S VII.3, Proposition 3]{bourb_int_7-9_en} as the desired maximal compact lift. This shows that \Cref{item:thn:char.conn.alm.ell.lie:glob} and \Cref{item:thn:char.conn.alm.ell.lie:loc} are indeed equivalent, so it will not matter which phrasing is used. 

  We induct on the length $\mathrm{max}$ of the closed derived series
  \begin{equation*}
    \bG_{sol}=
    \bG_{sol}^{((0))}
    \ge
    \bG_{sol}^{((1))}
    \ge
    \cdots
    \ge
    \bG_{sol}^{((\mathrm{max}))}
    \gneq
    \{1\}
  \end{equation*}
  with \Cref{th:no.triv.wt} providing the abelian-radical base case.
  
  \begin{enumerate}[label={},wide]
  \item \textbf{\Cref{item:thn:char.conn.alm.ell.lie:glob}/\Cref{item:thn:char.conn.alm.ell.lie:loc}: $\left(\Leftarrow\right)$} \Cref{pr:gmodl2g} delivers the induction step, taking into account the fact that \emph{compact} normal subgroups do not affect almost-ellipticity by \Cref{le:alm.ell.perm}\Cref{item:le:alm.ell.perm:quot.by.cpct}. 
    
  \item \textbf{\Cref{item:thn:char.conn.alm.ell.lie:glob}/\Cref{item:thn:char.conn.alm.ell.lie:loc}: $\left(\Rightarrow\right)$} The induction hypothesis reduces the problem to proving that if $\bG$ is almost-elliptic, then $\bK$ operates on the non-compact bottom layer $\left(\fg^{((\mathrm{max}-1))}_{sol}/\fg^{((\mathrm{max}))}_{sol}\right)_{nc}$ without trivial weights. This, in turn, follows from \Cref{th:no.triv.wt,th:ext.by.rd.perm}.

  \item \textbf{\Cref{item:thn:char.conn.alm.ell.lie:auto.opn}} We can assume via \Cref{le:alm.ell.perm}\Cref{item:le:alm.ell.perm:quot.by.cpct} that
    \begin{equation*}
      \bG_{sol}^{((\mathrm{max}))}
      \cong
      \bR^d\times \bT^e
      \quad
      \left(\text{\cite[Theorem 4.2.4]{de}}\right)
    \end{equation*}
    is in fact a vector group by quotienting out its unique maximal compact subgroup $\bT^e$. The induction step then hypothesizes $\bG/\bG_{sol}^{((\mathrm{max}))}$ openly almost-elliptic with maximal compact subgroups $\bK\le \bG$ acting trivial-weight-freely on that vector group. Now, the elements of $\bG$ which
    \begin{itemize}[wide]
    \item have elliptic image in $\bG/\bG_{sol}^{((\mathrm{max}))}$;

    \item and have no 1-eigenvectors in $\bG_{sol}^{((\mathrm{max}))}$
    \end{itemize}
    are elliptic by \Cref{th:no.triv.wt} (or \Cref{item:thn:char.conn.alm.ell.lie:glob} of the present statement), and constitute an open dense set by the induction hypothesis.  \qedhere
  \end{enumerate}
\end{thn:char.conn.alm.ell.lie}

\begin{remark}\label{re:solv.bad.exp}
  In comparing and passing back-and-forth between \Cref{thn:conn.cpct.by.ss.solv,th:no.triv.wt} on the one hand and \Cref{thn:char.conn.alm.ell.lie} on the other, (non-)simple connectedness raises some subtleties that in might be worthwhile to note. 
  
  The universal cover $\widetilde{\bG}\xrightarrowdbl{}\bG$ of a connected Lie group $\bG$ as in \Cref{thn:char.conn.alm.ell.lie} will be of the form $\bL\rtimes \bK$, thus fitting into the framework of \Cref{thn:conn.cpct.by.ss.solv} with connected, simply-connected, solvable $\bL$ and compact, connected, simply-connected, semisimple $\bK$. We have
  \begin{equation*}
    \bG
    \cong
    \left(\bL\rtimes \bK\right)/\widetilde{\bD}
    =
    \widetilde{\bG}/\widetilde{\bD}
    ,\quad
    \text{central discrete }\widetilde{\bD}\le \widetilde{\bG},
  \end{equation*}
  and the intersection $\bD:=\widetilde{\bD}\cap \bL$ is central discrete in $\bL$ as well as $\bK$-invariant. 
  
  In general, there will be infinitely many choices for an abelian Lie subalgebra $\bR^d\cong \fa\le \fl:=Lie(\bL)$ that will ``linearize'' a discrete central subgroup $\bD\le \bL$ of a solvable simply-connected Lie group $\bL$ in the sense that
  \begin{equation*}
    \bD=\exp\left(\text{some lattice in }\fa\right). 
  \end{equation*}
  This is in evidence in the examples usually given \cite[Exercise II.B.4]{helg_dglgssp} to illustrate exponential-function pathologies (non-bijectivity) for simply-connected solvable Lie groups: said example is $\bL\cong \bR^2\rtimes \bR$ for an action
  \begin{equation}\label{eq:r.by.rot}
    \bR
    \xrightarrowdbl{\quad}
    \bR/\bZ
    \cong
    \bS^1
    \cong
    \mathrm{SO}(2)\le \mathrm{GL}(\bR^2).
  \end{equation}
  The center of $\bL$ consists precisely of the kernel $\bZ$ of \Cref{eq:r.by.rot}, and all conjugates of the non-normal semidirect second factor $\bR\le \bL$ contain it. It should be clear how to construct ample supplies of such examples: set $\bL\cong \bR^s\rtimes \bR^d$ for actions $\bR^d \circlearrowright \bR^s$ that factor through torus quotients $\bR^d\xrightarrowdbl{}\bT^d$.

  Despite the multiplicity of $\fa$ containing $\bD$ as a lattice, there will always be (at least) one such that is in addition $\bK$-invariant: \Cref{pr:lc.alm.conn.max.cpct.inv}.  
\end{remark}


The following observation is by now a simple matter of unwinding some of the (otherwise non-trivial) structure theory for locally compact groups. 



\begin{proposition}\label{pr:lc.alm.conn.max.cpct.inv}
  A compact group $\bK$ acting continuously on a locally compact almost-connected group $\bG$ leaves invariant some maximal compact subgroup thereof. 
\end{proposition}
\begin{proof}
  Consider the (still locally compact, almost-connected) semidirect product $\widetilde{\bG}:=\bG\rtimes \bK$. By \cite[Lemma 3.10(i)]{zbMATH00540631} it has a maximal compact subgroup $\widetilde{\bM}$, which may as well be assumed to contain $\bK$; it follows that
  \begin{equation*}
    \widetilde{\bM}
    =
    \bM\cdot \bK
    \cong
    \bM\rtimes \bK
    ,\quad
    \bM:=\widetilde{\bM}\cap \bG,
  \end{equation*}
  and I claim that $\bM$ is maximal compact in $\bG$; the conclusion will then follow from the $\bK$-invariance of $\bM$. It thus remains to verify the claimed maximality of the compact subgroup $\bM\le \bG$. The only hypotheses needed will be
  \begin{itemize}[wide]
  \item the normality of $\bG\trianglelefteq\widetilde{\bG}$ (the semidirect-product decomposition of the larger group being irrelevant);

  \item and the fact that $\widetilde{\bM}\le \widetilde{\bG}$ is maximal compact with $\widetilde{\bG}=\bG\cdot \widetilde{\bM}$. 
  \end{itemize}

  
  Recall first \cite[\S 4.6, Theorem]{mz} that $\widetilde{\bG}$ is \emph{pro-lie} in the sense of \cite[\S 1]{hm_pro-lie-struct}: the inverse limit of its Lie quotients; in particular $\bG$ is the inverse limit of its Lie quotients $\bK$-equivariantly, so it will suffice, for our present purposes, to assume $\bG$ Lie to begin with. Now,
  \begin{equation}\label{eq:grd}
    \bG/\bM
    \quad\cong\quad
    \bG\cdot \widetilde{\bM}/\widetilde{\bM}
    =
    \widetilde{\bG}/\widetilde{\bM}
    \quad\cong\quad
    \bR^d
    \quad
    \left(\text{homeomorphism}\right)
  \end{equation}
  for some $d\in \bZ_{\ge 0}$ by \cite[Theorem C]{zbMATH00540631} (applicable to almost-connected locally compact groups, as observed in the course of the proof of \cite[Lemma 3.10]{zbMATH00540631}). On the other hand, let $\overline{\bM}\le \bG$ be maximal compact containing $\bM$ (one such does exist by \cite[Lemma 3.10(i)]{zbMATH00540631}). The same \cite[Theorem C]{zbMATH00540631} gives
  \begin{equation}\label{eq:gre}
    \bG/\bM
    \quad\cong\quad
    \bG/\overline\bM\times \overline{\bM}/\bM
    \quad\cong\quad
    \bR^e\times \left(\text{compact manifold }\overline{\bM}/\bM\right)
    ,\quad
    e\in \bZ_{\ge 0}. 
  \end{equation}
  Compact boundary-less $(d-e)$-manifolds have non-trivial $\bZ/2$-valued cohomology in the top degree $d-e$ \cite[Corollary VI.7.14]{bred_gt_1997}, so \Cref{eq:grd,eq:gre} jointly imply that $\bM=\overline{\bM}$. 
\end{proof}

\begin{remarks}\label{res:inv.cpct.sbgps}
  \begin{enumerate}[(1),wide]
  \item There are many variations on the theme of \Cref{pr:lc.alm.conn.max.cpct.inv} in the literature: 
    \begin{itemize}[wide]
    \item Per \cite[\S 6, Lemme]{SB_1956-1958__4__73_0} (or \cite[Exercise VI.A.8(i)]{helg_dglgssp}, or \cite[Proposition 2.14]{hm_pro-lie-struct}), compact automorphism groups of semisimple Lie groups leave invariant maximal compact subgroups thereof. The first two sources rely on a fixed-point result for actions on negatively-curved Riemannian manifolds \cite[Theorem I.13.5]{helg_dglgssp}. 

    \item Cf. also \cite[Corollary X.1]{hm_split}, ``straightening'' a direct product decomposition $\bR^d\times \left(\text{compact}\right)$ into a $\bK$-invariant one whenever a compact group $\bK$ operates on such a product.
      
    \end{itemize}

  \item The case $\bG\cong V\rtimes \bM$ with $V\cong \bR^d$ and the compact group $\bM$ carrying $\bK$-actions and $\bM$ is represented on $V$ $\bK$-equivariantly can be cast as an instance of cohomology vanishing for compact groups.

    denote the $\bK$-actions on $V$ and $\bM$ by left-hand superscripts. Defining
    \begin{equation*}
      \tensor*[^s]{(0,g)}{}
      =:
      \left(v_{s,g},\tensor*[^s]{g}{}\right)
      ,\quad
      s\in \bK
      ,\quad
      v_{\bullet}\in V
      \quad\text{and}\quad
      g\in \bM,
    \end{equation*}
    a straightforward computation shows that
    \begin{equation*}
      \bK
      \ni s
      \xmapsto{\quad\varphi\quad}
      \left(\bM\ni g
        \xmapsto{\quad}
        v_{\tensor*[^s]{g}{},s}
        \in V
      \right)    
    \end{equation*}
    is a \emph{1-cocycle} \cite[\S 5.1]{ser_gal}
    \begin{equation*}
      \varphi
      \in
      Z^1\left(\bK,\ Z^1\left(\bM,\ V\right)\right),
    \end{equation*}
    with $\bK$ acting on the space of (continuous) 1-cocycles $\bM\to V$ by
    \begin{equation*}
      \tensor*[^s]{\psi}{}
      :=
      \tensor*[^s]{\left(\psi(\tensor*[^{s^{-1}}]{\bullet}{})\right)}{}
      ,\quad
      s\in \bK
      ,\quad
      \psi\in Z^1(\bM,V). 
    \end{equation*}
    The usual \cite[I, Theorem 2.3]{moore-ext} Haar averaging technique for compact-group cohomology shows that cocycles are coboundaries, so there must be some $v\in V$ with 
    \begin{equation*}
      \forall\left(s\in \bK\right)
      \forall\left(g\in \bM\right)
      \quad:\quad
      \tensor*[^s]{(v-gv,\ g)}{}
      =
      \left(\tensor*[^s]{v}{}-\tensor*[^s]{g}{}\tensor*[^s]{v}{},\ \tensor*[^s]{g}{}\right);
    \end{equation*}
    or: the (maximal compact) conjugate
    \begin{equation*}
      \Ad_{(v,1)}\bM
      =
      \left\{\left(v-gv,\ g\right)\ :\ g\in \bM\right\}
      \le
      V\rtimes \bM
    \end{equation*}
    of $\bM\le V\rtimes \bM$ is $\bK$-invariant.
  \end{enumerate}
\end{remarks}

We observed in the course of the proof of \Cref{pr:lc.alm.conn.max.cpct.inv} that
\begin{equation}\label{eq:int.norm.mc}
  \left.
    \begin{aligned}
      \bG&\quad\text{locally compact almost-connected}\\
      \bH&\trianglelefteq \bG\quad\text{normal}\\
      \bM&\le \bG\quad\text{maximal compact}
    \end{aligned}
  \right\}
  \quad\xRightarrow{\quad}\quad
  \bM\cap \bH\le \bH\quad\text{maximal compact}
\end{equation}
The analogous statement holds \cite[Lemma 3.10(iv)]{zbMATH00540631} for closed, possibly non-normal $\bH\le \bG$ containing $\bG_0$. We record in \Cref{le:norm.by.g0}, for completeness and future reference, a common generalization. 

\begin{lemma}\label{le:norm.by.g0}
  Let $\bG$ be a locally compact almost-connected group and $\bM\le \bG$ maximal compact.

  Whenever the closed subgroup $\bH\le \bG$ is normalized by the identity component $\bG_0$, the intersection $\bH\cap \bM\le \bH$ is maximal compact. 
\end{lemma}
\begin{proof}
  Apply \cite[Lemma 3.10(iv)]{zbMATH00540631} to the $\bG_0$-containing closed subgroup $\widetilde{\bH}:=\overline{\bG_0\cdot \bH}$ to conclude that $\bM\cap \widetilde{\bH}\le \widetilde{\bH}$ is maximal compact, and then again \Cref{eq:int.norm.mc} with $\widetilde{\bH}$ (where $\bH$ is normal) in place of $\bG$. 
\end{proof}

\begin{remark}\label{re:bagley.ea}
  The remarks preliminary to \cite[Lemma 3.10]{zbMATH00540631} refer to \cite[Lemma 1.8]{zbMATH00025545} as a slight generalization of \cite[Lemma 3.10(iv)]{zbMATH00540631}, but note that the former result deduces the maximality of $\bM\cap \bH\le \bH$ among compact \emph{normal} subgroups provided $\bH\trianglelefteq \bG$ is closed normal and $\bM\le \bG$ is maximal compact and again normal. So this is yet another variant of the result. 
\end{remark}

\addcontentsline{toc}{section}{References}

\def\polhk#1{\setbox0=\hbox{#1}{\ooalign{\hidewidth
  \lower1.5ex\hbox{`}\hidewidth\crcr\unhbox0}}}

\Addresses

\end{document}